\documentclass[reqno, 12pt]{amsart}
\usepackage{amssymb,amsmath,hyperref}
\usepackage{amsrefs, float}
\usepackage[foot]{amsaddr}
\usepackage{bbold,stackrel, multicol}
 
\newcommand{\Su}{\mathcal{S}}
\newcommand{\Var}{{\bf Var}}

\newtheorem{thm}{Theorem}

\newtheorem{cor}[thm]{Corollary}
\newtheorem{defi}[thm]{Definition}
\newtheorem{claim}[thm]{Claim}

\newtheorem{nota}[thm]{Notation}

\newtheorem*{tempo*}{Template}
\newtheorem{theorem}[thm]{Theorem}

\newtheorem{lemma}[thm]{Lemma}
\newtheorem{definition}[thm]{Definition}

\newcommand\be{\begin{equation}}
\newcommand\ee{\end{equation}} 

\usepackage{amsmath,amsfonts} 

\usepackage[applemac]{inputenc}

\def\bdefi{\begin{defi}\rm}
\def\edefi{\end{defi}}
\def\bnota{\begin{nota}\rm}
\def\enota{\end{nota}}

\def\SIX{\Pi_{2}^{1}\text{-\textsf{\textup{CA}}}_{0}}

\def\ZFC{\textup{\textsf{ZFC}}}

\def\({\textup{(}}
\def\){\textup{)}}

\def\bye{\end{document}}
\def\N{{\mathbb  N}}
\def\Q{{\mathbb  Q}}
\def\R{{\mathbb  R}}

\def\SS{\textup{\textsf{S}}}
\def\LL{{\mathfrak{L}}}

\def\J{\mathcal{J}}
\def\D{{\mathcal  D}}

\def\di{\rightarrow}

\def\asa{\leftrightarrow}

\def\SUP{\textup{\textsf{sup}}}

\def\NIN{\textup{\textsf{NIN}}}

\def\NBI{\textup{\textsf{NBI}}}

\def\w{\textup{\textsf{w}}}
\def\BW{\textup{\textsf{BW}}}

\def\eps{\varepsilon}

\def\cc{\mathcal{C}}

\usepackage{graphicx}
\usepackage{tikz}
\usetikzlibrary{matrix, shapes.misc}
\usepackage{comment,tikz-cd}

\numberwithin{equation}{section}
\numberwithin{thm}{section}

\begin{document}
\title{Betwixt Turing and Kleene}
\author{Dag Normann}
\address{Department of Mathematics, The University 
of Oslo, P.O. Box 1053, Blindern N-0316 Oslo, Norway}
\email{dnormann@math.uio.no}
\author{Sam Sanders}
\address{Department of Philosophy II, RUB Bochum, Germany}
\email{sasander@me.com}
\begin{abstract}
Turing's famous `machine' model constitutes the first intuitively convincing framework for \emph{computing with real numbers}.  
Kleene's computation schemes S1-S9 extend Turing's approach and provide a framework for \emph{computing with objects of any finite type}.  
Various research programs have been proposed in which higher-order objects, like functions on the real numbers, are \emph{represented/coded} as real numbers, so as to make 
them amenable to the Turing framework.  It is then a natural question whether there is any significant difference between the Kleene approach or the Turing-approach-via-codes.  
Continuous functions being well-studied in this context, we study \emph{functions of bounded variation}, which have \textbf{at most countably} many points of discontinuity. 
A central result is the \emph{Jordan decomposition theorem} that a function of bounded variation on $[0,1]$ equals the difference of two monotone functions.  
We show that for this theorem and related results, the difference between the Kleene approach and the Turing-approach-via-codes is \emph{huge}, in that full second-order arithmetic readily comes to the fore in Kleene's approach, in the guise of Kleene's quantifier $\exists^{3}$. 
\end{abstract}


\maketitle
\thispagestyle{empty}

\section{Introduction: Jordan, Turing, and Kleene}\label{intro}
In a nutshell, we study the computational properties of the \emph{Jordan decomposition theorem} as in Theorem \ref{drd} and other results on \emph{functions of bounded variation}, establishing the huge differences between the \emph{Turing and Kleene approaches} to computability theory. 
For the rest of this section, we introduce the above italicised notions and sketch the contents of this paper in more detail.  All technical notions are introduced in Section \ref{prelim} while our main results are in Section \ref{main}.  

\medskip

First of all, Turing's famous `machine' model, introduced in \cite{tur37}, constitutes the first intuitively convincing framework for \emph{computing with real numbers}.  
Kleene's computation schemes S1-S9, introduced in \cite{kleeneS1S9} extend Turing's framework and provide a framework for \emph{computing with objects of any finite type}.  
Now, various\footnote{Examples of such frameworks include: reverse mathematics (\cites{simpson2, stillebron}), constructive analysis (\cite{beeson1}*{I.13}, \cite{bish1}), predicative analysis (\cite{littlefef}), and computable analysis (\cite{wierook}). Note that Bishop's constructive analysis is not based on Turing computability \emph{directly}, but one of its `intended models' is however (constructive) recursive mathematics, as discussed in \cite{brich}.  One aim of Feferman's predicative analysis is to capture constructive reasoning in the sense of Bishop.} research programs have been proposed in which higher-order objects are \emph{represented/coded} as real numbers or similar second-order representations, so as to make 
them amenable to the Turing framework.  It is then a natural question whether there is any significant difference\footnote{The \emph{fan functional} constitutes an early \emph{natural} example of this difference: it has a computable code but is not S1-S9 computable (but S1-S9 computable in Kleene's $\exists^{2}$ from Section \ref{prelim1}).  The fan functional computes a modulus of uniform continuity for continuous functions on Cantor space; details may be found in \cite{longmann}.\label{seeyouwell}} between the Kleene approach or the Turing-approach-via-codes.     
Continuous functions being well-studied$^{\ref{seeyouwell}}$ in this context, we investigate \emph{functions of bounded variation}, which have \textbf{at most} countably many points of discontinuity.

\medskip

Secondly, the notion of \emph{bounded variation} was first introduced by Jordan around 1881 (\cite{jordel}) yielding a generalisation of Dirichlet's convergence theorems for Fourier series.  
Indeed, Dirichlet's convergence results are restricted to functions that are continuous except at a finite number of points, while functions of bounded variation can have (at most) countable many points of discontinuity, as also shown by Jordan, namely in \cite{jordel}*{p.\ 230}.
The fundamental theorem about functions of bounded variation is as follows and can be found in \cite{jordel}*{p.\ 229}.
\begin{thm}[Jordan decomposition theorem]\label{drd}
A function $f : [0, 1] \di \R$ of bounded variation can be written as the difference of two non-decreasing functions $g, h:[0,1]\di \R$.
\end{thm}
The computational properties of Theorem \ref{drd} have been studied extensively via second-order representations, namely in e.g.\ \cites{groeneberg, kreupel, nieyo, verzengend}.
The same holds for constructive analysis by \cites{briva, varijo,brima, baathetniet}, involving different (but related) constructive enrichments.  
Now, finite iterations of the Turing jump suffice to compute $g, h$ from Theorem \ref{drd} in terms of \emph{represented} functions $f$ of bounded variation by \cite{kreupel}*{Cor.\ 10}.

\medskip

Thirdly, in light of the previous, it is a natural question what the computational properties of Theorem \ref{drd} are in Kleene's framework.  
In particular, the following question is central to this paper.
\begin{center}
\emph{How hard is it to compute \(S1-S9\) from $f:[0,1]\di \R$ of bounded variation, two monotone functions $g, h$ such that $f=g-h$ on $[0,1]$?}
\end{center}
A functional that can perform this computational task will be called a \emph{Jordan realiser}, introduced in Definition \ref{JDR}.  A related and \emph{natural} computational task is as follows.
\begin{center}
\emph{How hard is it to compute \(S1-S9\) from $f:[0,1]\di \R$ of bounded variation, the supremum $\SUP_{x\in [0,1]}f(x)$?}
\end{center}
A functional that can perform this computational task will be called a \emph{\textsf{\textup{sup}}-realiser} (see Definition \ref{JDR}).
This task restricted to \emph{continuous} functions is well-studied, and rather weak by \cite{kohlenbach2}*{Footnote 6}.  In light of the above, the following computational task is also natural:
\begin{center}
\emph{How hard is it to compute \(S1-S9\) from $f:[0,1]\di \R$ of bounded variation, a sequence $(x_{n})_{n\in \N}$ listing the points of discontinuity of $f$?}
\end{center} 
By way of a robustness result, we show that the above three tasks are \emph{the same} modulo Kleene's $\exists^{2}$ from Section \ref{prelim1}.
Moreover, we show that Jordan realisers are \emph{hard} to compute: no type two functional, in particular 
 the functionals $\SS_{k}^{2}$ which decide $\Pi_{k}^{1}$-formulas (see Section~\ref{prelim1}), can compute a Jordan realiser.
We also show that Jordan realisers are \emph{powerful}: when combined with other natural functionals, one can go all the way up to Kleene's quantifier $\exists^{3}$ which yields full second-order arithmetic (see again Section \ref{prelim1} for the definition of $\exists^{3}$).  We also study special cases of Jordan realisers, which connects to the computational tasks associated to the \emph{uncountability of $\R$}.

\medskip

Finally, our main results are obtained in Section \ref{main} while some preliminary notions, including some essential parts of Kleene's higher-order computability theory, can be found in Section \ref{prelim}.

\section{Preliminaries}\label{prelim}
\subsection{Kleene's higher-order computability theory}\label{kle9}
We first make our notion of `computability' precise as follows.  
\begin{enumerate}
\item[(I)] We adopt $\ZFC$, i.e.\ Zermelo-Fraenkel set theory with the Axiom of Choice, as the official metatheory for all results, unless explicitly stated otherwise.
\item[(II)] We adopt Kleene's notion of \emph{higher-order computation} as given by his nine clauses S1-S9 (see \cite{longmann}*{Ch.\ 5} or \cite{kleeneS1S9}) as our official notion of `computable'.
\end{enumerate}
We refer to \cite{longmann} for a thorough overview of higher-order computability theory.
We do mention the distinction between `normal' and `non-normal' functionals  based on the following definition from \cite{longmann}*{\S5.4}. 
\bdefi\label{norma}
For $n\geq 2$, a functional of type $n$ is called \emph{normal} if it computes Kleene's $\exists^{n}$ following S1-S9, and \emph{non-normal} otherwise.  
\edefi
\noindent
We only make use of $\exists^{n}$ for $n=2,3$, as defined in Section \ref{prelim1}.

\medskip

It is a historical fact that higher-order computability theory, based on Kleene's S1-S9, has focused primarily on the world of \emph{normal} functionals (see \cite{longmann}*{\S5.4} for this opinion).  
We have previous studied the computational properties of new \emph{non-normal} functionals, namely those that compute the objects claimed to exist by:
\begin{itemize}
\item the Heine-Borel and Vitali covering theorems (\cites{dagsam, dagsamII, dagsamVI}),
\item the Baire category theorem (\cite{dagsamVII}),
\item local-global principles like \emph{Pincherle's theorem} (\cite{dagsamV}),
\item the uncountability of $\R$ and the Bolzano-Weierstrass theorem for countable sets in Cantor space (\cites{dagsamX, dagsamXI}),
\item weak fragments of the Axiom of (countable) Choice (\cite{dagsamIX}).
\end{itemize}
In this paper, we continue this study for the Jordan decomposition theorem and other basic properties of functions of bounded variation.
Next, we introduce some required higher-order notions in Section~\ref{klonkio}.

\subsection{Some higher-order notions}\label{klonkio}
\subsubsection{Some higher-order functionals}\label{prelim1}
We introduce a number of comprehension functionals from the literature.  
We are dealing with \emph{conventional} comprehension, i.e.\ only parameters over $\N$ and $\N^{\N}$ are allowed in formula classes like $\Pi_{k}^{1}$ et cetera.  

\medskip

First of all, the functional $\varphi$ as in $(\exists^{2})$ is clearly discontinuous at $f=11\dots$; in fact, $(\exists^{2})$ is equivalent to the existence of $F:\R\di\R$ such that $F(x)=1$ if $x>0$, and $0$ otherwise (see \cite{kohlenbach2}*{\S3}).  
\be\label{muk}\tag{$\exists^{2}$}
(\exists \varphi^{2}\leq_{2}1)(\forall f^{1})\big[(\exists n^{0})(f(n)=0) \asa \varphi(f)=0    \big]. 
\ee
Intuitively speaking, the functional $\varphi$ from $(\exists^{2})$ can decide the truth of any $\Sigma_{1}^{0}$-formula in its (Kleene) normal form.  
Related to $(\exists^{2})$, the functional $\mu^{2}$ in $(\mu^{2})$ is also called \emph{Feferman's $\mu$} (\cite{avi2}), defined as follows:
\begin{align}\label{mu}
(\exists \mu^{2})(\forall f^{1})\big[ (\exists n)(f(n)=0) \di &[f(\mu(f))=0\wedge (\forall i<\mu(f))(f(i)\ne 0) ]\notag\\
& \wedge [ (\forall n)(f(n)\ne0)\di   \mu(f)=0]    \big]. \tag{$\mu^{2}$}
\end{align}
We have $(\exists^{2})\asa (\mu^{2})$ over a weak system by \cite{kooltje}*{Prop.\ 3.4 and Cor.~3.5}) while $\mu^{2}$ is readily computed from $\varphi^{2}$ in $(\exists^{2})$.
The third-order functional from $(\exists^{2})$ is also called `Kleene's quantifier $\exists^{2}$', and we use the same convention for other functionals.  


\medskip

Secondly, $\SS^{2}$ as in $(\SS^{2})$ is called \emph{the Suslin functional} (\cite{kohlenbach2, avi2}):
\be\tag{$\SS^{2}$}
(\exists\SS^{2}\leq_{2}1)(\forall f^{1})\big[  (\exists g^{1})(\forall n^{0})(f(\overline{g}n)=0)\asa \SS(f)=0  \big].
\ee
Intuitively, the Suslin functional $\SS^{2}$ can decide the truth of any $\Sigma_{1}^{1}$-formula in its normal form.  
We similarly define the functional $\SS_{k}^{2}$ which decides the truth or falsity of $\Sigma_{k}^{1}$-formulas (again in normal form).  
We note that the operators $\nu_{n}$ from \cite{boekskeopendoen}*{p.\ 129} are essentially $\SS_{n}^{2}$ strengthened to return a witness to the $\Sigma_{n}^{1}$-formula at hand.  
As suggested by its name, $\nu_{k}$ is the restriction of Hilbert-Bernays' $\nu$ from \cite{hillebilly2}*{p.\ 495} to $\Sigma_{k}^{1}$-formulas. 
We sometimes use $\SS_{0}^{2}$ and $\SS_{1}^{2}$ to denote $\exists^{2}$ and $\SS^{2}$.  

\medskip

\noindent
Thirdly, second-order arithmetic is readily derived from the following:
\be\tag{$\exists^{3}$}
(\exists E^{3}\leq_{3}1)(\forall Y^{2})\big[  (\exists f^{1})(Y(f)=0)\asa E(Y)=0  \big].
\ee
The functional from $(\exists^{3})$ is also called `Kleene's quantifier $\exists^{3}$'.  
Hilbert-Bernays' $\nu$ from \cite{hillebilly2}*{p.\ 495} trivially computes $\exists^{3}$.

\medskip

Finally, the functionals  $\SS_{k}^{2}$ are defined using the usual formula class $\Pi_{k}^{1}$, i.e.\ only allowing first- and second-order parameters.  
We have dubbed this the \emph{conventional approach} and the associated functionals are captured by the umbrella term \emph{conventional comprehension}. 
Comprehension involving third-order parameters has previously (only) been studied in \cites{littlefef, kohlenbach4}, to the best of our knowledge.

\subsubsection{Some higher-order definitions}\label{defki}
We introduce some required definitions, which are all standard.

\medskip

First of all, a fruitful and faithful approach is the representation of sets by characteristic functions (see e.g.\ \cites{samnetspilot, samcie19,samwollic19,dagsamVI,dagsamVII, kruisje, dagsamIX}), well-known from e.g.\ measure and probability theory.   We shall use this approach, always assuming $\exists^{2}$ to make sure open sets represented by countable unions of basic opens  are indeed sets in our sense.  

\medskip

Secondly, we now turn to \emph{countable} sets.
Of course, the notion of `countable set' can be formalised in various ways, as follows.  
\bdefi[Enumerable set]\label{eni}
A set $A\subset \R$ is \emph{enumerable} if there is a sequence $(x_{n})_{n\in \N}$ such that $(\forall x\in \R)(x\in A\asa (\exists n\in \N)(x=x_{n}))$.  
\edefi
Definition \ref{eni} reflects the notion of `countable set' from reverse mathematics (\cite{simpson2}*{V.4.2}).  
%
Our definition of `countable set' is as follows.  
\bdefi[Countable set]\label{standard}~
A set $A\subset \R$ is \emph{countable} if there is $Y:\R\di \N$ such that 
\be\label{polki}
(\forall x, y\in A)(Y(x)=_{0}Y(y)\di x=y).
\ee 
The functional $Y$ as in \eqref{polki} is called \emph{injective} on $A$ or \emph{an injection} on $A$.
If $Y:\R\di \N$ is also \emph{surjective}, i.e.\ $(\forall n\in \N)(\exists x\in A)(Y(x)=n)$, we call $A$ \emph{strongly countable}.
The functional $Y$ is then called \emph{bijective} on $A$ or \emph{a bijection} on $A$.
\edefi
The first part of Definition \ref{standard} is from Kunen's set theory textbook (\cite{kunen}*{p.~63}) and the second part is taken from Hrbacek-Jech's set theory textbook \cite{hrbacekjech} (where the term `countable' is used instead of `strongly countable').  According to Veldman (\cite{veldje2}*{p.\ 292}), Brouwer studies set theory based on injections in \cite{brouwke}.
Hereonafter, `strongly countable' and `countable' shall exclusively refer to Definition~\ref{standard}.  

\section{Main results}\label{main}
In this section, we shall obtain our main results as follows.  Recall that a \emph{Jordan realiser} outputs the monotone functions claimed to exist by the Jordan decomposition theorem as in Theorem \ref{drd}. 
\begin{itemize}
\item We introduce Jordan realisers and other functionals witnessing basic properties of functions of bounded variation; we show that three of these are computationally equivalent (Section \ref{seccer}).  
\item  We show that Jordan realisers are \emph{hard} to compute based on results from \cite{dagsamX} (Section \ref{seccer2}).
\item  We show that Kleene's $\exists^{3}$ can be computed  from $\exists^{2}$, a Jordan realiser, and a well-ordering of $[0,1]$ (Section~\ref{seccer3}).
\item We show that Jordan realisers remain hard to compute even if we severely restrict the output (Section \ref{seccer4}).  

\end{itemize}

\subsection{Jordan realisers and equivalent formulations}\label{seccer}
We introduce functionals witnessing basic properties of functions of bounded variation, including the Jordan decomposition theorem (Theorem \ref{drd}).
We show that three of these are computationally equivalent given $\exists^2$.  

\medskip

As noted above, we always assume $\exists^2$ but specify the use when essential. 
This means that we can use the concept of Kleene-computability over $\R$ or $[0,1]$ without focusing on how these spaces are represented.

\medskip

First of all, as to definitions, the \emph{total variation} of a function $f:[a, b]\di \R$ is (nowadays) defined as follows:
\be\label{tomb}\textstyle
V_{a}^{b}(f):=\sup_{a\leq x_{0}< \dots< x_{n}\leq b}\sum_{i=0}^{n} |f(x_{i})-f(x_{i+1})|.
\ee
If this quantity exists and is finite, one says that $f$ has bounded variation on $[a,b]$.
Now, the notion of bounded variation is defined in \cite{nieyo} \emph{without} mentioning the supremum in \eqref{tomb}; this approach can also be found in \cite{kreupel, briva, brima}.  Hence, we shall hereafter distinguish between the following two notions.  As it happens, Jordan seems to use item \eqref{donp} of Definition \ref{varvar} in \cite{jordel}*{p.\ 228-229}, providing further motivation for the functionals introduced in Definition \ref{JDR}.
\bdefi[Variations on variation]\label{varvar}
\begin{enumerate}  
\renewcommand{\theenumi}{\alph{enumi}}
\item The function $f:[a,b]\di \R$ \emph{has bounded variation} on $[a,b]$ if there is $k_{0}\in \N$ such that $k_{0}\geq \sum_{i=0}^{n} |f(x_{i})-f(x_{i+1})|$ 
for any partition $x_{0}=a <x_{1}< \dots< x_{n-1}<x_{n}=b  $.\label{donp}
\item The function $f:[a,b]\di \R$ \emph{has {a} variation} on $[a,b]$ if the supremum in \eqref{tomb} exists and is finite.\label{donp2}
\end{enumerate}
\edefi
We can now introduce the following notion of `realiser' for the Jordan decomposition theorem and related functionals.
\begin{definition}\label{JDR}\rm~
\begin{itemize}
\item A \emph{Jordan realiser} is a partial functional $\J$ of type 3 taking as input a function $f:[0,1] \rightarrow \R$ which has bounded variation (item \eqref{donp} in Definition \ref{varvar}), and providing a pair $(g,h)$ of increasing functions $g$ and $h$ such that $f = g-h$ on $[0,1]$.
\item A \emph{weak Jordan realiser} is a partial functional $\J_{\w}$ of type 3 taking as inputs a function $f:[0,1] \rightarrow \R$ and its bounded variation $V_{0}^{1}(f)$ (item \eqref{donp2} in Definition \ref{varvar}), and providing a pair $(g,h)$ of increasing functions $g$ and $h$ such that $f = g-h$ on $[0,1]$.
\item A \emph{$\SUP$-realiser} is a partial functional $\mathcal{S}$ of type 3 taking as input a function $f:[0,1] \rightarrow \R$ which has bounded variation (item \eqref{donp} in Definition \ref{varvar}), and providing the supremum $\SUP_{x\in [0,1]}f(x)$. 
\item A \emph{continuity-realiser} is a partial functional $\mathcal{L}$ of type 3 taking as input a function $f:[0,1] \rightarrow \R$ which has bounded variation (item \eqref{donp} in Definition \ref{varvar}), and providing a sequence $(x_{n})_{n\in \N}$ which lists all points of discontinuity of $f$ on $[0,1]$. 
\end{itemize}
\end{definition}
Next, we need the following lemma.  The use of $\exists^{2}$ is perhaps superfluous in light of the constructive proof in \cite{tognog}, but the latter seems to make essential use of the Axiom of (countable) Choice.   
\begin{lemma}\label{korf} 
There is a functional $\D$, computable in $\exists^2$, such that if $f:[0,1] \rightarrow \R$ is increasing \(decreasing\), then $\D(f)$ enumerates all points of discontinuity of $f$ on $[0,1]$. 
\end{lemma}
\begin{proof}
Let $\{q_i\}_{i \in \N}$ be an enumeration of $\Q \cap [0,1]$, let $f:[0,1]\di \R$ be monotone, and define $a_i := f(q_i)$. 
Let $A$ be the set of pairs of  rationals $p < r$ such that there is no $i$ with $p < a_i < r$. 
For $(p,r) \in A$, define 
\be\label{deco}
x_{p,r} := \sup\{q_i : a_i \leq p\}  = \inf \{q_j : r \leq a_j\}.
\ee
It is easy to see that the reals in \eqref{deco} are equal; indeed, the existence of a rational $q_i$ between them, together with assuming that $(p,r) \in A$, leads to a contradiction.

\medskip

Then all discontinuities of $f$ will be among the elements $x_{p,r}$  for $(p,r) \in A$. Clearly, $i \mapsto a_i$, $A$ and $(p,r) \mapsto x_{p,r}$ for $(p,r) \in A$ are computable in $f$ and $\exists^2$. This ends the proof.
\end{proof}
Finally, we show that three `non-weak' realisers from Definition \ref{JDR} are in fact one and the same, in part based on Lemma \ref{korf}. 
\begin{thm}\label{thm.surprise}
Assuming $\exists^{2}$, Jordan realisers, $\SUP$-realisers, and continuity realisers are computationally equivalent.  
\end{thm}
\begin{proof}
We first show that a continuity realiser computes a $\SUP$-realiser. 
To this end, let $f$ be of bounded variation on $[0,1]$ and let $\mathcal{L}(f)=(x_n)_{n \in \N}$ be a list of its points of discontinuity.  From $\mathcal{L}(f)$ we can find a list $(y_j)_{j \in \N}$ containing both the points of discontinuity and all rational numbers in $[0,1]$. Then we can compute 
\be\label{gandalf}
\mathcal{S}(f) = \sup\{f(x) : x \in [0,1]\} = \sup\{f(y_j) : j \in \N\},
\ee
since $\sup\{f(y_j) : j \in \N\}$ is computable from $\mathcal{L}(f)$, $f$, and $\exists^2$. 

\medskip

Secondly, that a Jordan realiser computes a continuity realiser, assuming $\exists^{2}$, is immediate from Lemma \ref{korf}. 

\medskip

Thirdly, we show that a $\SUP$-realiser $\Su$ computes a Jordan realiser.
To this end, let $a$ and $b$ be such that $0 \leq a < b \leq 1$ and define
\be\label{fereng}
\textup{$\Su^+_{a,b}(f) := \sup_{x\in [a,b]}f(x) $ and $\Su^-_{a,b}(f) := \inf_{x\in [a,b]}f(x)$}.
\ee
These functionals are clearly computable in $\Su$, for $f:[0,1]\di \R$ a function of bounded variation.  
Now let $\Var(P,f)$ be the sum $\sum_{i=0}^{n-1} |f(x_{i+1})-f(x_{i})|$ for a partition $P = \{0 = x_0 < \cdots <  x_n = 1\}$ of $[0,1]$, while $\Var^{+}(P,f)$ is the sum of all \emph{positive} differences $f(x_{i+1}) -f(x_i)$. 
\begin{claim}\label{claim1} 
To compute a Jordan realiser, it suffices to compute $\Delta(f):=\sup_{P} \Var^+(P,f)$, where $P$ varies over all partitions of $[0,1]$.
\end{claim}
To prove Claim \ref{claim1}, we compute increasing functions $f^+$ and $f^-$ such that $f = f^+ - f^-$. Without loss of generality, we may assume that $f(0) = 0$, and by symmetry it suffices to compute $f^+$. 
We can define $f^+(0) = 0$ and $f^+(x) = \Delta(f_x)$ for $x > 0$, where $f_x(y) = f(\frac{y}{x})$. This ends the proof of Claim \ref{claim1}.

\medskip

\noindent We now employ the functionals $\Su^+_{a,b}$ and $\Su^-_{a,b}$ from \eqref{fereng} as follows. 
\begin{definition}\rm Let $f$ be of bounded variation on $[0,1]$ and  let $n \in \N$.
An \emph{$n,f$-trail} is a sequence $k_0, \ldots , k_{2m-1}$ such that $0 \leq k_0 < \ldots  < k_{2m-1} < n$ and such that when $0 \leq j < m$ we have
\[
\Su^-_{\frac{k_{2j}}{n}, \frac{k_{2j}+1}{n}}(f) < \Su^+_{\frac{k_{2j+1}}{n},\frac{k_{2j+1}+1}{n}}(f).
\]
Define the positive $n$-\emph{move}  $\bar M(n,f)$ as the maximal value of 
\[\textstyle
\sum_{j = 0}^{m-1}\Big(\Su^+_{\frac{k_{2j+1}}{n},\frac{k_{2j+1}+1}{n}}(f) - \Su^-_{\frac{k_{2j}}{n}, \frac{k_{2j}+1}{n}}(f)\Big)
\]
where $k_0, \ldots ,k_{2m-1}$ varies over all $n,f$-trails. 
\end{definition}
%
%
\begin{claim}\label{claim2}
For each $n\in \N$, we have that $\bar M(n,f) \leq \Delta(f)$.
\end{claim}
To prove Claim \ref{claim2}, it suffices to show that $\bar M(n,f) \leq \Delta(f) + \epsilon$ for each $\epsilon > 0$. Let $k_0, \ldots ,k_{2m-1}$ be an $n,f$-trail giving the value of  $\bar M(n,f)$. 
For each $j < m$,  select $x_j \in \big[\frac{k_{2j+1}}{n},\frac{k_{2j+1}+1}{n}\big]$ such that $f(x_k) > \Su^+_{\frac{k_{2j+1}}{n},\frac{k_{2j+1}+1}{n}}(f) - \frac{\epsilon}{2m}$ 
and $y_j \in \big[\frac{k_{2j}}{n}, \frac{k_{2j}+1}{n}\big]$ such that $f(y_k) <  \Su^-_{\frac{k_{2j}}{n}, \frac{k_{2j}+1}{n}}(f) + \frac{\epsilon}{2m}$.
For any partition $P$ containing all points $x_j$ and $y_j$ we have that $\Var^+(P,f) > \bar M(n,f) - \epsilon$.  As a result, we obtain 
\[
\bar M(n,f) < \Var^+(P,f) + \epsilon \leq \Delta(f) + \epsilon,
\] 
which ends the proof of Claim \ref{claim2}.
%
%
\begin{claim}\label{claim3}
Let $f$ be of bounded variation on $[0,1]$ and let $P$ be a partition of $[0,1]$. Then there is $n \in \N$ such that $\Var^+(P,f) \leq \bar M(n,f)$.
\end{claim}
%
To prove Claim \ref{claim3}, fix $P = \{0 \leq s_0  <  \cdots < s_{m'}\leq 1\}$. Without loss of generality, we may assume that $f(s_0) < f(s_1) > f(s_2) < \cdots < f(s_{m'})$, i.e. the values $f(s_i)$ alternate between going up and going down,  so $m'$ is an uneven number $2m-1$. Let $n$ be such that each $[\frac{k}{n}, \frac{k+1}{n}]$, for $k < n$, contains at most one $s_i$ and each $s_i$ is contained in exactly one $[\frac{k_i}{n}, \frac{k_i+1}{n}]$. 
Then $k_0 , \ldots ,k_{2m-1}$ is an $n,f$-trail witnessing that $\Var^+(P,f) \leq \bar M(n,f)$.
This ends the proof of Claim \ref{claim3}.  

\medskip

Finally, by the above three claims, we have $\Delta(f) = \sup_{n\in \N}\bar M(n,f)$.  
Since we can compute the latter from $\Su$ and $\exists^2$, the former is likewise computable. 
As we can compute a Jordan realiser from $\Delta$, the proof of Theorem \ref{thm.surprise} is complete.
\end{proof}
In conclusion, as perhaps expected in light of \eqref{gandalf}, rather effective and pointwise approximation results exist for functions of bounded variation \emph{at points of continuity} (see e.g.\ \cite{chin}*{p. 261}).  
For points of discontinuity, it seems one only approximates the average of the left and right limits, i.e.\ not the function value itself.

\subsection{Jordan realiser and countable sets}\label{seccer2}
We show that Jordan realisers are \emph{hard} to compute by connecting them to computability theoretic results pertaining to \emph{countable sets} from \cite{dagsamX}.
Recall the definitions from Section \ref{defki} pertaining to the latter notion.

\medskip

Now, the most fundamental property of countable sets is that they can be enumerated, i.e.\ listed as a sequence, as explicitly noted by e.g.\ Borel in \cite{opborrelen2} in his early discussions of the Heine-Borel theorem.
Next, we show that Jordan realisers can indeed enumerate countable sets as in Definition \ref{standard}. 
\begin{thm}\label{flungo}
Together with $\exists^{2}$, a Jordan realiser $\J$ can perform the following computational procedures.
\begin{itemize}
\item Given a set $A\subset [0,1]$ and $Y:[0,1]\di \N$ injective on $A$, produce a sequence $(x_{n})_{n\in \N}$ listing exactly the elements of $A$.
\item Given $F:[0,1]\di [0,1]$, $A\subset [0,1]$, and $Y:[0,1]\di \N$ injective on $A$, produce $\sup_{x\in A}F(x)$.
\end{itemize}
\end{thm}
\begin{proof}
We only need to establish the first item, as the second item readily follows from the first one using $\exists^{2}$.
Let $A\subset [0,1]$ be countable, i.e.\ there is $Y:[0,1]\di \N$ which is injective on $A$.  
Use $\exists^{2}$ to define the function $f:\R\di \R$ defined as follows:
\be\label{klamda}
f(x):=
\begin{cases}
\frac{1}{2^{Y(x)+1}} & x\in A\\
0 &\textup{ otherwise }
\end{cases}.
\ee
Following item \eqref{donp} in Definition \ref{varvar}, the function $f$ has bounded variation on $[0,1]$ as any sum $\sum_{i=0}^{n} |f(x_{i})-f(x_{i+1})|$ is at most $\sum_{i=0}^{n} \frac{1}{2^{i+1}}$ for $x_{i}$ in $[0,1]$ and $i\leq n+1$.
Now let $\J(f)=(g, h)$ be such that $f=g-h$ on $[0,1]$ and recall $\D$ from Lemma \ref{korf}.  Use $\exists^{2}$ to define the sequence $(x_{n})_{n\in \N}$ as all the reals in $\D(g)$ and $\D(h)$.  
Now consider the following formula for any $x\in [0,1]$:
\be\label{flamda}
\big[(\exists n\in \N)(x= x_{n})\wedge g(x)\ne h(x)\big] \asa x\in A.
\ee
The forward direction in \eqref{flamda} is immediate as $g(x)\ne h(x)$ for $x\in [0,1]$ implies $f(x)>0$, and hence $x\in A$ by definition. 
For the reverse direction, fix $x\in A$ and note that $0<f(x)=g(x)-h(x)$ by the definition of $f$ in \eqref{flamda}, i.e.\ $g(x)\ne h(x)$ holds.  
Moreover, in case $(\forall n\in \N)(x\ne x_{n})$, then $g$ and $h$ are continuous at $x$, by the definition of $(x_{n})_{n\in \N}$. 
Hence, $f$ is continuous at $x$, which is only possible if $f(x)=0$, but the latter implies $x\not \in A$, by \eqref{klamda}, a contradiction.
In this way, we obtain \eqref{flamda} and we may enumerate $A$ by removing from $(x_{n})_{n\in \N}$ all elements not in $A$, which can be done using $\exists^{2}$.  
\end{proof}
Secondly, weak Jordan realisers can enumerate \emph{strongly} countable sets (Definition \ref{standard}).
\begin{cor}\label{corref}
Together with $\exists^{2}$, a weak Jordan realiser $\J_{\w}$ can perform the following computational procedures.
\begin{itemize}
\item Given a set $A\subset [0,1]$ and $Y:[0,1]\di \N$ \textbf{bijective} on $A$, produce a sequence $(x_{n})_{n\in \N}$ listing exactly the elements of $A$.
\item Given $F:[0,1]\di [0,1]$, $A\subset [0,1]$, and $Y:[0,1]\di \N$ \textbf{bijective} on $A$, produce $\sup_{x\in A}F(x)$.
\end{itemize}
\end{cor}
\begin{proof}
Following item \eqref{donp2} in Definition \ref{varvar}, the function $f$ in \eqref{klamda} has total variation exactly $1$ in case $Y$ is additionally a bijection. 
\end{proof}
A weak Jordan realiser cannot compute a Jordan realiser; this remains true if we combine the former with an arbitrary type 2 functional.  
Since the proof of this claim is rather lengthy, 
 we have omitted the former from this paper.  

\medskip

Thirdly, the functional $\Omega_{\BW}$ introduced and studied in \cite{dagsamX}*{\S4}, performs the computational procedure from the second item in Theorem~\ref{flungo}, 
leading to the following corollary. 
\begin{cor}
Together with the Suslin functional $\SS^{2}$, a Jordan realiser $\J$ computes $\SS^{2}_{2}$, i.e.\ a realiser for $\SIX$. 
\end{cor}
\begin{proof}
By \cite{dagsamX}*{Theorem 4.6.(b)}, $\Omega_{\BW}+ \SS^{2}$ computes $\SS_{2}^{2}$, while a Jordan realiser $\J$ computes $\Omega_{\BW}$ by Theorem \ref{flungo}.
\end{proof}
Finally, under the additional set-theoretic hypothesis $\textsf{V=L}$, the combination $\Omega_{\BW}+ \SS^{2}$ even computes $\exists^{3}$ by \cite{dagsamX}*{Theorem 4.6.(c)}.
An obvious question is whether a similar result can be obtained within $\ZFC$, which is the topic of the following section.  

\subsection{Computing Kleene's $\exists^{3}$ from Jordan realisers}\label{seccer3}
We show that Kleene's quantifier $\exists^{3}$ is computable in the combination of:
\begin{itemize}
\item Kleene's quantifier $\exists^2$,
\item any Jordan realiser $\J$ (or: $\Omega_{\BW}$ from Section \ref{seccer2}),
\item a well-ordering $\prec$ of $[0,1]$.  
\end{itemize}
We note that the third item exists by the Axiom of Choice.  Assuming $\prec$ and $\preceq$ are the irreflexive and reflexive versions of the same well-ordering of $[0,1]$, they are computable in each other and $\exists^2$.
\begin{theorem}\label{thm.first}
Let $\J$ be a Jordan realiser. Then $\exists^3$ is Kleene-computable in $\J$, $\prec$, and $\exists^2$.
\end{theorem}
\begin{proof}
We actually prove a slightly stronger result. Let $\Omega$ be a partial functional of type 3 such that $\Omega(X)$ terminates whenever $X\subset \R$ has at most one element, and $\Omega(X) \in X$ whenever $X$ contains exactly one element. 
One readily\footnote{To obtain an enumeration of $A\subset [0,1]$ given $Y:[0,1]\di \R$ injective on $A$, define $E_{n}:=\{x\in A: Y(x)=n\}$ and define $x_{n}:= \Omega(E_{n})$ in case the latter is in $E_{n}$, and $0$ otherwise.} shows that $\Omega$ is computationally equivalent to $\Omega_{\rm BW}$, given $\exists^{2}$. We now show that $\exists^3$ is computable in $\Omega$, $\prec$, and $\exists^2$. 

\medskip

We let $x,y$ vary over $[0,1]$ and we fix $h:[0,1] \rightarrow \{0,1\}$. We aim to compute $\exists^3(h)$ by deciding the truth of the formula $(\exists x\in [0,1]) (h(x) = 1)$. 
 To this end, consider the functionals $E_{\prec x}$ and $E_{\preceq x}$ defined as:
\[
\textup{$E_{\prec x}(h) = 1 \leftrightarrow (\exists y \prec x) (h(y) = 1)$ and $E_{\preceq x}(h) = 1 \leftrightarrow (\exists y \preceq x) (h(y) = 1)$}.
\]
We shall show that these are computable in $\Omega$ and $\exists^2$, uniformly in $x$. Note that $E_{\preceq x}$ is trivially computable in $E_{\prec x}$, uniformly in $x$, so we settle for computing $E_{\prec x}(h)$. The argument will be by the recursion theorem, so we give the algorithm for computing $E_{\prec x}(h)$ using $x$, $h$, $\Omega$  and $E_{\prec y}$ for $y \prec x$.
Now let $x$ and $h$ be fixed and define $h_x$ as:
\[
h_{x}(y):=
\begin{cases}
0 & {\rm if }~  x \preceq y \\ 
h(y) & {\rm if }~ y \prec x \wedge (\forall z \prec y) (h(z) = 0)  \\ 
0 & {\rm otherwise}  
\end{cases},
 \]
where we use $E_{\prec y}$ to decide whether the second case holds. Then $h_x$ is constant zero if $(\forall y \prec x) (h(y) = 0)$, and if not, $h_x$ takes the value 1 in exactly the \emph{least} point $y \prec x$ where $h(y) = 1$. Hence, we have
\[
(\exists y \prec x)(h(y) = 1) \leftrightarrow [\Omega(h_x) \prec x \wedge h(\Omega(h_x)) = 1].
\]
We now apply the recursion theorem to find an index $e\in \N$ such that for all $x\in [0,1]$ and $h:[0,1]\di \{0,1\}$ (and well-orderings $\preceq)$, we have 
\[
\{e\}(\Omega, h,x,\preceq)  \simeq E_{\preceq x}(h).
\]  
Then we use transfinite induction over $\preceq$ to prove that $e\in \N$ defines a total functional doing what it is supposed to do. 
For readers not familiar with this use of the recursion theorem, what we do is defining $\{e_0\}(d,\Omega,h,x,\preceq)$ as in the construction, but replacing all uses of $E_{\preceq y}(h')$ by $ \{d\}(\Omega,h',y,\preceq)$;
we then use the fact that there is an index $e$ such that $\{e\}(\dots)\simeq  \{e_0\}(e,\dots)$, where `$\dots$' are the other parameters.

\medskip

Having established the computability of each $E_{\prec x}$ from $x$, we can use the same trick to compute $\exists^3(h)$ as follows: construct from $h$ a function $h'$ that takes the value 1 in at most one place, namely the $\prec$-least $x$ where $h(x) = 1$, in case such exists. 
\end{proof}
Finally, we note that Kleene's computation scheme S9 is essentially a `hard-coded' version of the recursion theorem for S1-S9, while S1-S8 merely define (higher-order) primitive recursion.  
In this way, the recursion theorem is central to S1-S9, although we have previously witnessed S1-S9 computations via primitive recursive terms.  

\subsection{Jordan realisers and the uncountability of $\R$}\label{seccer4}
We show that a number of interesting functionals, including `heavily restricted' Jordan realisers,
are (still) quite hard to compute, based on the computational properties of the uncountability of $\R$ pioneered in \cite{dagsamX}. 

\medskip

First of all, in more detail, Theorem \ref{thm.surprise} implies that a Jordan realiser can enumerate all points of discontinuity of a function of bounded variation. 
It is then a natural question whether Jordan realisers remain hard to compute if we only require the output to be e.g.\ \emph{one} point of continuity.  
By way of an answer, Theorem \ref{floeper} lists a number of interesting functionals -including the aforementioned `one-point' Jordan realisers- that compute 
functionals witnessing the uncountability of $\R$.  Functionals related to the uncountability of $\R$ are special in the following sense.  

\medskip

By the previous sections, Jordan realisers have surprising properties and are a nice addition to the pantheon of interesting non-normal functionals stemming from ordinary mathematics (see \cites{dagsam, dagsamII, dagsamIII, dagsamV, dagsam VII, dagsamVI, dagsamX, dagsamIX, dagsamXI} or Section \ref{kle9} for other examples).  It is then a natural question what the \emph{weakest} such functional is; a candidate is provided by the \emph{uncountability of $\R$}, 
which can be formulated in various guises as follows.
\begin{itemize} 
\item Cantor's theorem: there is no surjection from $\N$ to $\R$.
\item $\NIN$: there is no injection from $[0,1]$ to $\N$.
\item $\NBI$: there is no bijection from $[0,1]$ to $\N$.
\end{itemize}
Cantor's theorem is provable in constructive and computable mathematics (\cites{bish1, simpson2}), while there is even an efficient algorithm to compute from a sequence of reals, a real not in that sequence (\cite{grayk}). 
As explored in \cite{dagsamX}, $\NIN$ and $\NBI$ are hard to prove in terms of conventional comprehension.   
We will not study $\NBI$ in this paper while Cantor's theorem and $\NIN$ give rise to the following specifications.  
\bdefi[Realisers for the uncountability of $\R$]~
\begin{itemize}
\item A \emph{Cantor functional/realiser} takes as input $A\subset [0,1]$ and $Y:[0,1]\di \N$ such that $Y$ is injective on $A$, and outputs $x\not \in A$.  
\item A \emph{\textbf{weak} Cantor realiser} takes as input $A\subset [0,1]$ and $Y:[0,1]\di \N$ such that $Y$ is \textbf{bijective} on $A$, and outputs $x\not \in A$.  
\item A $\NIN$-\emph{realiser} takes as input $Y:[0,1]\di \N$ and outputs $x,y\in [0,1]$ with $x\ne y \wedge Y(x)=Y(y)$.  
\end{itemize}
\edefi
As explored in \cite{dagsamX}, $\NIN$-realisers are among the weakest non-normal functionals originating from ordinary mathematics we have studied. 
Moreover, one readily\footnote{Let $N$ be a $\NIN$-realiser and let $A\subset [0,1]$ and $Y:[0,1]\di \N$ be such that $Y$ is injective on $A$.  Define $Z:[0,1]\di \N$ as follows: $Z(x):=Y(x)+1$ in case $x\in A$, and $0$ otherwise.  Clearly, $N(Z)(1)\not \in A$ as required for a Cantor functonal.} proves that a $\NIN$-realiser computes a Cantor realiser, while the latter are still hard to compute as follows.  
\begin{thm}\label{forgukel}
No type 2 functional computes a weak Cantor realiser.  
\end{thm}
\begin{proof}
Fix some functional $F^{2}$ and assume wlog that $F$ computes $\exists^{2}$.
Assume there is a Cantor realiser $\cc$ computable in $F$.
Now let $A$ be the set of reals computable in $F$ and define $Y:[0,1]\di \N$ as follows using \emph{Gandy selection} (see \cite{dagsamIX, longmann} for an introduction): 
for the first part, define $Y(x)$ as an index for computing $x$ from $F$ in case $x \in A$; we put $Y(x):=0$ in case $x\not \in A$. 
By assumption, $\cc(A,Y)$ terminates as $Y$ is injective on $A$. 
Since the restriction of $Y$ to $A$ is partially computable in $F$, all oracle calls of the form `$x \in A$' will be answered with yes, since $x$ then is computable in $F$.  Hence, all oracle calls for the value $Y(x) $ can be answered \emph{computably in $F$}.  In this way, $\cc(A, Y)$ is computable in $F$, which also follows from \cite{dagsamX}*{Lemma 2.15}. 
Hence, $\cc(A,Y) \in A$ by definition, a contradiction.  The proof remains valid if we extend $A$ to some $B\subset [0,1]$ and extend $Y$ to $Z:[0,1]\di \N$ bijective on $B$.
\end{proof}
Secondly, it is fairly trivial to prove (classically) that there is no \emph{continuous} injection from $[0,1]$ to $\Q$, based on the intermediate value theorem.
Now consider the following principle, which expresses a \emph{very special case} of the uncountability of $\R$.
\begin{itemize} 
\item $\NIN_{\textsf{\textup{BV}}}$: there is no injection from $[0,1]$ to $\Q$ that has \emph{bounded variation} (item \eqref{donp} in Definition \ref{varvar}).  
\end{itemize}
One readily establishes the equivalence $\NIN\asa \NIN_{\textsf{\textup{BV}}}$ over a weak system, following the proof of Theorem \ref{floeper}. 
By the latter and Theorem~\ref{forgukel}, while $\NIN$-realisers are defined for all $Y:[0,1]\di \N$, the restriction to functions of bounded variation, which only have countably many points of discontinuity, is (still) hard to compute \emph{and} intermediate between Cantor and $\NIN$-realisers. 

\medskip

Thirdly, we have the following theorem where the functional 
\[\textstyle
\LL(f)(s):=\int_{0}^{+\infty} e^{-st}f(t)~dt 
\] 
is the \emph{Laplace transform} of $f:\R\di \R$.
Since we restrict to functions of bounded variation, we interpret $\LL(f)$ as the limit of Riemann integrals, if the latter exists.  
It is well-known that if $\LL(f)$ and $\LL(g)$ exists and are equal everywhere, $f$ and $g$ are equal almost everywhere, inspiring the final -considerably weaker- item in Theorem \ref{floeper}.
In the below items, `bounded variation' refers to item \eqref{donp} of Definition~\ref{varvar}.
\begin{thm}\label{floeper}
Assuming $\exists^{2}$, a Cantor realiser can be computed from a functional performing any of the following tasks.
\begin{itemize}
\item For $f:[0,1]\di \Q$ which has bounded variation, find $x, y\in [0,1]$ such that $x\ne y$ and $f(x)= f(y)$. 
\item For $f:[0,1]\di \R$ which has bounded variation, find a point of continuity in $[0,1]$.
\item If $f:[0,1]\di [0, 1]$ is Riemann integrable \(or has bounded variation\) with $\int_{0}^{1}f(x)~dx =0$, find $ y\in [0,1]$ with $f(y)=0$.  
\item If $f, g:\R\di \R$ satisfy the following:
\begin{itemize}
\item $f, g$ have bounded variation on $[0, a]$ for any $a\in \R^{+}$, 
\item $\LL(f)$ and $\LL(g)$ exists and are equal on $[0, +\infty)$, 
\end{itemize}
find $x\in (0, \infty)$ with $f(x)=g(x)$.
\end{itemize}
Any $\NIN$-realiser computes a functional as in the first item.  Assuming $\exists^{2}$, a \textbf{weak} Cantor realiser can be computed from a functional performing any of the above tasks restricted as in item \eqref{donp2} of Definition~\ref{varvar}..  
\end{thm}
\begin{proof}
The penultimate sentence is immediate as $\Q$ and $\N$ are bijective.  
Now fix some countable set $A\subset [0,1]$ and $Y:[0,1]\di \N$ injective on $A$.  
Consider $f:[0,1]\di \R$ as in \eqref{klamda} and recall it has bounded variation.  
For the second item, if $x_{0}\in [0,1]$ is a point of continuity of $f$, we must have $f(x_{0})=0$ by continuity.  
Then $x_{0}\not\in A $ by definition, as required.   
For the first item, in case  $x\ne y$ and $f(x)= f(y)$, we must have $x\not \in A$ or $y\not \in A$, in light of \eqref{klamda}.

\medskip

For the third item, consider $f:[0,1]\di \R$ as in \eqref{klamda}; that $\int_{0}^{1}f(x)~dx $ exists and equals $0$, follows from the usual $\eps$-$\delta$-definition of Riemann integrability. 
Indeed, fix $\eps_{0}>0$ and find $k_{0}\in \N$ such that $\frac{1}{2^{k_{0}}}<\eps_{0}$.  
Let $P$ be a partition $x_{0}:=0, x_{1}, \dots, x_{n}, x_{n+1}:=1$ of $[0,1]$ with $t_{i}\in [x_{i}, x_{i+1}]$ for $i\leq n$ and with mesh $\|P\|:=\max_{i\leq n }(x_{i+1}-x_{i})$ at most $ \frac{1}{2^{k_{0}}}$.  
Then the Riemann sum $S(f, P):=\sum_{i\leq n}f(t_{i})(x_{i+1}-x_{i})$ satisfies
\[\textstyle
S(f, P)\leq \frac{1}{2^{k_{0}}} \sum_{i\leq n}f(t_{i})\leq \frac{1}{2^{k_{0}}}\sum_{i\leq n}\frac{1}{2^{i+1}}\leq \frac{1}{2^{k_{0}}}, 
\]
as $Y$ is injective on $A$ and $f$ is zero outside of $A$.
Hence, $\int_{0}^{1}f(x)~dx =0$ and any $y\in [0,1]$ with $f(y)=0$ yields $y\not \in A$ by \eqref{klamda}.

\medskip

For the final item, the tangent and arctangent functions provide bijections between $(0,1)$ and $\R$.  Hence, we may work with $A\subset \R$ and $Y:\R\di \N$ injective on $A$.  
We again consider $f$ as in \eqref{klamda}, now as an $\R\di \R$-function.  This function $f$ has bounded variation on any interval $[0, a]$ for $a>0$, in the same way as in the proof of Theorem \ref{flungo}.
Since $e^{-z}\leq 1$ for $z\geq0$, the previous paragraph yields $\int_{0}^{N} e^{-st}f(t)~dt =0$ for any $N\in \N, s \geq 0$, and hence $\LL(f)(s)=0$ for all $s\geq 0$.
For $g:\R\di \R$ the zero everywhere function, we trivially have $\LL(g)(s)=0$ for any $s\geq 0$.  
Clearly, any $x\in [0, +\infty)$ such that $f(x)=g(x)=0$ also satisfies $x\not \in A$, yielding a Cantor functional. 
The final sentence now follows in light of (the proof of) Corollary \ref{corref}.
\end{proof}
We note that theorem also goes through if we formulate the second item using the much weaker notions of {quasi}\footnote{A function $ f:X\rightarrow \mathbb {R} $ is \emph{quasi-continuous} (resp.\ \emph{cliquish}) at $x\in X$ if for any $ \epsilon > 0$ and any open neighbourhood $U$ of $x$, there is a non-empty open ball ${ G\subset U}$ with $(\forall y\in G) (|f(x)-f(y)|<\eps)$ (resp.\ $(\forall y, z\in G) (|f(z)-f(y)|<\eps)$).\label{fulgsde}} or cliquish$^{\ref{fulgsde}}$ continuity in at least one point in $[0,1]$.  These notions are found in e.g.\ \cites{nieuwebron, kowalski}.


\begin{bibdiv}
\begin{biblist}
\bib{avi2}{article}{
  author={Avigad, Jeremy},
  author={Feferman, Solomon},
  title={G\"odel's functional \(``Dialectica''\) interpretation},
  conference={ title={Handbook of proof theory}, },
  book={ series={Stud. Logic Found. Math.}, volume={137}, },
  date={1998},
  pages={337--405},
}

\bib{beeson1}{book}{
  author={Beeson, Michael J.},
  title={Foundations of constructive mathematics},
  series={Ergebnisse der Mathematik und ihrer Grenzgebiete},
  volume={6},
  publisher={Springer},
  date={1985},
  pages={xxiii+466},
}

\bib{bish1}{book}{
  author={Bishop, Errett},
  title={Foundations of constructive analysis},
  publisher={McGraw-Hill},
  date={1967},
  pages={xiii+370},
}

\bib{opborrelen2}{book}{
  author={Borel, E.},
  title={Le\c {c}ons sur la th\'eorie des fonctions},
  year={1898},
  publisher={Gauthier-Villars, Paris},
  pages={pp.\ 136},
}

\bib{brich}{book}{
  author={Bridges, Douglas},
  author={Richman, Fred},
  title={Varieties of constructive mathematics},
  series={London Mathematical Society Lecture Note Series},
  volume={97},
  publisher={Cambridge University Press},
  place={Cambridge},
  date={1987},
  pages={x+149},
}

\bib{briva}{article}{
  author={Bridges, Douglas},
  title={A constructive look at functions of bounded variation},
  journal={Bull. London Math. Soc.},
  volume={32},
  date={2000},
  number={3},
  pages={316--324},
}

\bib{brima}{article}{
  author={Bridges, Douglas},
  author={Mahalanobis, Ayan},
  title={Bounded variation implies regulated: a constructive proof},
  journal={J. Symbolic Logic},
  volume={66},
  date={2001},
  number={4},
  pages={1695--1700},
}

\bib{tognog}{inproceedings}{
  author={Bridges, Douglas},
  editor={Kosheleva, Olga and Shary, Sergey P. and Xiang, Gang and Zapatrin, Roman},
  title={Constructive Continuity of Increasing Functions},
  booktitle={Beyond Traditional Probabilistic Data Processing Techniques: Interval, Fuzzy etc. Methods and Their Applications},
  year={2020},
  publisher={Springer},
  pages={9--19},
}

\bib{brouwke}{article}{
  author={Brouwer, L. E. J.},
  title={Begr\"undung der Mengenlehre unabh\"angig vom logischen Satz vom ausgeschlossenen Dritten. Erster Teil: Allgemeine Mengenlehre},
  journal={Koninklijke Nederlandsche Akademie van Wetenschappen, Verhandelingen, 1ste sectie},
  volume={12},
  number={5},
  date={1918},
  pages={pp.\ 43},
}

\bib{boekskeopendoen}{book}{
  author={Buchholz, Wilfried},
  author={Feferman, Solomon},
  author={Pohlers, Wolfram},
  author={Sieg, Wilfried},
  title={Iterated inductive definitions and subsystems of analysis},
  series={LNM 897},
  publisher={Springer},
  date={1981},
  pages={v+383},
}

\bib{chin}{article}{
  author={Ch\^{e}ng, Fu Hua},
  title={On the rate of convergence of Bernstein polynomials of functions of bounded variation},
  journal={J. Approx. Theory},
  volume={39},
  date={1983},
  number={3},
  pages={259--274},
}

\bib{littlefef}{book}{
  author={Feferman, Solomon},
  title={How a Little Bit goes a Long Way: Predicative Foundations of Analysis},
  year={2013},
  note={unpublished notes from 1977-1981 with updated introduction, \url {https://math.stanford.edu/~feferman/papers/pfa(1).pdf}},
}

\bib{grayk}{article}{
  author={Gray, Robert},
  title={Georg Cantor and transcendental numbers},
  journal={Amer. Math. Monthly},
  volume={101},
  date={1994},
  number={9},
  pages={819--832},
}

\bib{groeneberg}{article}{
  title={Highness properties close to PA-completeness},
  author={Noam Greenberg},
  author={Joseph S. Miller},
  author={Andr\'e Nies},
  year={2019},
  journal={To appear in Israel Journal of Mathematics},
}

\bib{baathetniet}{article}{
  author={Heyting, Arend},
  title={Recent progress in intuitionistic analysis},
  conference={ title={Intuitionism and Proof Theory}, address={Proc. Conf., Buffalo, N.Y.}, date={1968}, },
  book={ publisher={North-Holland, Amsterdam}, },
  date={1970},
  pages={95--100},
}

\bib{hillebilly2}{book}{
  author={Hilbert, David},
  author={Bernays, Paul},
  title={Grundlagen der Mathematik. II},
  series={Zweite Auflage. Die Grundlehren der mathematischen Wissenschaften, Band 50},
  publisher={Springer},
  date={1970},
}

\bib{hrbacekjech}{book}{
  author={Hrbacek, Karel},
  author={Jech, Thomas},
  title={Introduction to set theory},
  series={Monographs and Textbooks in Pure and Applied Mathematics},
  volume={220},
  edition={3},
  publisher={Marcel Dekker, Inc., New York},
  date={1999},
  pages={xii+291},
}

\bib{jordel}{article}{
  author={Jordan, Camillie},
  title={Sur la s\'erie de Fourier},
  journal={Comptes rendus de l'Acad\'emie des Sciences, Paris, Gauthier-Villars},
  volume={92},
  date={1881},
  pages={228--230},
}

\bib{kleeneS1S9}{article}{
  author={Kleene, Stephen C.},
  title={Recursive functionals and quantifiers of finite types. I},
  journal={Trans. Amer. Math. Soc.},
  volume={91},
  date={1959},
  pages={1--52},
}

\bib{kohlenbach4}{article}{
  author={Kohlenbach, Ulrich},
  title={Foundational and mathematical uses of higher types},
  conference={ title={Reflections on the foundations of mathematics}, },
  book={ series={Lect. Notes Log.}, volume={15}, publisher={ASL}, },
  date={2002},
  pages={92--116},
}

\bib{kooltje}{article}{
  author={Kohlenbach, Ulrich},
  title={On uniform weak K\"onig's lemma},
  note={Commemorative Symposium Dedicated to Anne S. Troelstra (Noordwijkerhout, 1999)},
  journal={Ann. Pure Appl. Logic},
  volume={114},
  date={2002},
  number={1-3},
  pages={103--116},
}

\bib{kohlenbach2}{article}{
  author={Kohlenbach, Ulrich},
  title={Higher order reverse mathematics},
  conference={ title={Reverse mathematics 2001}, },
  book={ series={Lect. Notes Log.}, volume={21}, publisher={ASL}, },
  date={2005},
  pages={281--295},
}

\bib{kowalski}{article}{
  author={Kowalewski, Marcin},
  author={Maliszewski, Aleksander},
  title={Separating sets by cliquish functions},
  journal={Topology Appl.},
  volume={191},
  date={2015},
  pages={10--15},
}

\bib{kreupel}{article}{
  author={Kreuzer, Alexander P.},
  title={Bounded variation and the strength of Helly's selection theorem},
  journal={Log. Methods Comput. Sci.},
  volume={10},
  date={2014},
  number={4},
  pages={4:16, 15},
}

\bib{kruisje}{article}{
  author={Kreuzer, Alexander P.},
  title={Measure theory and higher order arithmetic},
  journal={Proc. Amer. Math. Soc.},
  volume={143},
  date={2015},
  number={12},
  pages={5411--5425},
}

\bib{kunen}{book}{
  author={Kunen, Kenneth},
  title={Set theory},
  series={Studies in Logic},
  volume={34},
  publisher={College Publications, London},
  date={2011},
  pages={viii+401},
}

\bib{longmann}{book}{
  author={Longley, John},
  author={Normann, Dag},
  title={Higher-order Computability},
  year={2015},
  publisher={Springer},
  series={Theory and Applications of Computability},
}

\bib{nieuwebron}{article}{
  author={Neubrunn, T.},
  title={Quasi-continuity},
  journal={Real Anal. Exchange},
  volume={14},
  date={1988/89},
  number={2},
  pages={259--306},
}

\bib{nieyo}{article}{
  title={The reverse mathematics of theorems of Jordan and Lebesgue},
  journal={J. Sym. Logic},
  publisher={Cambridge University Press},
  author={Nies, Andr\'e},
  author={Triplett, Marcus},
  author={Yokoyama, Keita},
  year={2021},
  pages={1--18},
}

\bib{dagsam}{article}{
  author={Normann, Dag},
  author={Sanders, Sam},
  title={Nonstandard Analysis, Computability Theory, and their connections},
  journal={Journal of Symbolic Logic},
  volume={84},
  number={4},
  pages={1422--1465},
  date={2019},
}

\bib{dagsamII}{article}{
  author={Normann, Dag},
  author={Sanders, Sam},
  title={The strength of compactness in Computability Theory and Nonstandard Analysis},
  journal={Annals of Pure and Applied Logic},
  volume={170},
  number={11},
  date={2019},
}

\bib{dagsamIII}{article}{
  author={Normann, Dag},
  author={Sanders, Sam},
  title={On the mathematical and foundational significance of the uncountable},
  journal={Journal of Mathematical Logic, \url {https://doi.org/10.1142/S0219061319500016}},
  date={2019},
}

\bib{dagsamVI}{article}{
  author={Normann, Dag},
  author={Sanders, Sam},
  title={Representations in measure theory},
  journal={Submitted, arXiv: \url {https://arxiv.org/abs/1902.02756}},
  date={2019},
}

\bib{dagsamVII}{article}{
  author={Normann, Dag},
  author={Sanders, Sam},
  title={Open sets in Reverse Mathematics and Computability Theory},
  journal={Journal of Logic and Computation},
  volume={30},
  number={8},
  date={2020},
  pages={pp.\ 40},
}

\bib{dagsamV}{article}{
  author={Normann, Dag},
  author={Sanders, Sam},
  title={Pincherle's theorem in reverse mathematics and computability theory},
  journal={Ann. Pure Appl. Logic},
  volume={171},
  date={2020},
  number={5},
  pages={102788, 41},
}

\bib{dagsamX}{article}{
  author={Normann, Dag},
  author={Sanders, Sam},
  title={On the uncountability of $\mathbb {R}$},
  journal={Submitted, arxiv: \url {https://arxiv.org/abs/2007.07560}},
  pages={pp.\ 37},
  date={2020},
}

\bib{dagsamIX}{article}{
  author={Normann, Dag},
  author={Sanders, Sam},
  title={The Axiom of Choice in Computability Theory and Reverse Mathematics},
  journal={Journal of Logic and Computation},
  volume={31},
  date={2021},
  number={1},
  pages={297-325},
}

\bib{dagsamXI}{article}{
  author={Normann, Dag},
  author={Sanders, Sam},
  title={On robust theorems due to Bolzano, Weierstrass, and Cantor in Reverse Mathematics},
  journal={See \url {https://arxiv.org/abs/2102.04787}},
  pages={pp.\ 30},
  date={2021},
}

\bib{varijo}{article}{
  author={Richman, Fred},
  title={Omniscience principles and functions of bounded variation},
  journal={Mathematical Logic Quarterly},
  volume={48},
  date={2002},
  pages={111--116},
}

\bib{samcie19}{article}{
  author={Sanders, Sam},
  title={Nets and Reverse Mathematics: initial results},
  year={2019},
  journal={LNCS 11558, Proceedings of CiE19, Springer},
  pages={253-264},
}

\bib{samwollic19}{article}{
  author={Sanders, Sam},
  title={Reverse Mathematics and computability theory of domain theory},
  year={2019},
  journal={LNCS 11541, Proceedings of WoLLIC19, Springer},
  pages={550-568},
}

\bib{samnetspilot}{article}{
  author={Sanders, Sam},
  title={Nets and Reverse Mathematics: a pilot study},
  year={2021},
  journal={Computability},
  pages={31-62},
  volume={10},
  number={1},
}

\bib{simpson2}{book}{
  author={Simpson, Stephen G.},
  title={Subsystems of second order arithmetic},
  series={Perspectives in Logic},
  edition={2},
  publisher={CUP},
  date={2009},
  pages={xvi+444},
}

\bib{stillebron}{book}{
  author={Stillwell, J.},
  title={Reverse mathematics, proofs from the inside out},
  pages={xiii + 182},
  year={2018},
  publisher={Princeton Univ.\ Press},
}

\bib{tur37}{article}{
  author={Turing, Alan},
  title={On computable numbers, with an application to the Entscheidungs-problem},
  year={1936},
  journal={Proceedings of the London Mathematical Society},
  volume={42},
  pages={230-265},
}

\bib{veldje2}{article}{
  author={Veldman, Wim},
  title={Understanding and using Brouwer's continuity principle},
  conference={ title={Reuniting the antipodes}, address={Venice}, date={1999}, },
  book={ series={Synthese Lib.}, volume={306}, publisher={Kluwer}, },
  date={2001},
  pages={285--302},
}

\bib{wierook}{book}{
  author={Weihrauch, Klaus},
  title={Computable analysis},
  publisher={Springer-Verlag, Berlin},
  date={2000},
  pages={x+285},
}

\bib{verzengend}{article}{
  author={Zheng, Xizhong},
  author={Rettinger, Robert},
  title={Effective Jordan decomposition},
  journal={Theory Comput. Syst.},
  volume={38},
  date={2005},
  number={2},
  pages={189--209},
}


\end{biblist}
\end{bibdiv}
\bye